\documentclass[11pt]{amsart}
\usepackage{amssymb}

\numberwithin{equation}{section}

\theoremstyle{plain}
\newtheorem{theorem}{Theorem}[section]
\newtheorem{lem}[theorem]{Lemma}
\newtheorem{prop}[theorem]{Proposition}
\newtheorem{cor}[theorem]{Corollary}
\theoremstyle{definition}

\newcommand{\inn}[1]{ \langle #1 \rangle }
\newcommand{\R}{\mathbb{R}}

\newcommand{\Z}{\mathbb{Z}}
\newcommand{\Sp}{\mathbb{S}}

\newcommand{\M}{\mathcal{M}}
\newcommand{\Ny}{\mathcal{N}}

\newcommand{\cT}{\mathcal{T}}

\newcommand{\txi}{\tilde \xi}
\newcommand{\wh}{\widehat}
\newcommand{\om}{{\Omega}}
\newcommand{\cO}{{\mathcal{O}}}

\newcommand{\vp}{\varphi}

\newcommand{\les}{\lesssim}
\newcommand{\ges}{\gtrsim}

\newcommand{\norm}[1]{ || #1 ||}

\newcommand{\ind}{\mathbf{1}}
\newcommand{\sgn}{\operatorname{sgn}}

\newcommand{\dist}{\operatorname{dist}}

\AtBeginDocument{%
   \def\MR#1{}
}

\begin{document}
\title[Directional maximal functions]{Almost-orthogonality principles for certain directional maximal functions}
\author{Jongchon Kim}
\address{Department of Mathematics, University of British Columbia, 1984 Mathematics Road, Vancouver, BC, Canada V6T 1Z2}
\email{jongchon.kim.work@gmail.com}

\begin{abstract}
We develop almost-orthogonality principles for maximal functions associated with averages over line segments and directional singular integrals. Using them, we obtain sharp $L^2$-bounds for these maximal functions when the underlying direction set is equidistributed in $\Sp^{n-1}$.
\end{abstract}

\maketitle

\section{Introduction}
This paper is concerned with $L^2$-estimates for certain maximal functions associated with a set of direction $\om \subset \Sp^{n-1}$. For Nikodym and Kakeya maximal functions associated with averages over rectangles of bounded eccentricities, $L^2$-estimates are classical; see e.g. \cite{Cor, Str, Bo, Bar}. 

The first maximal function considered in this paper is associated with averages over line segments in a finite set of directions $\om\subset \Sp^{n-1}$: 
\[ M_{\om} f(x) = \sup_{v\in \om}  M_v f(x), \text{ where } \; M_v f(x) = \sup_{h>0} \frac{1}{2h}\int_{-h}^{h} |f(x-vt)| dt . \]
The second maximal function is a singular integral variant of $M_{\om}$. Suppose that $m \in C^\infty(\R\setminus \{0\})$ satisfies $|m^{(\alpha)}(\xi)| \leq C_\alpha |\xi|^{-\alpha}$ for all $\alpha\geq 0$. We consider a maximal function $T_\om$ associated with the directional singular integral $T_v$ given by $\widehat{T_v f}(\xi) = m(v\cdot \xi) \wh{f}(\xi)$ :
\[ T_\om f (x)= \sup_{v\in \om} |T_v f(x)|.\]
When $m(\xi) = -i\sgn(\xi)$, $T_v$ is the directional Hilbert transform. We shall denote by $H_\om$ the maximal function $T_\om$ associated with this particular $m$.

The main goal of this paper is to develop almost-orthogonality principles for $M_\om$ and $T_\om$. They quantify the contribution to the $L^2$-operator norm of these maximal operators from different parts of the direction set $\om$ and facilitate a divide and conquer argument. In $\R^2$, such results for $M_\om$ were obtained by Alfonseca, Soria, and Vargas \cite{ASV,ASV2,Alf}. We develop weaker versions for $M_\om$ and $T_\om$ which work in every dimension. As a corollary, we obtain sharp $L^2$-estimates for these maximal operators when $\om$ is equidistributed.

We say that $\om \subset \Sp^{n-1}$ is equidistributed if there is $0<\delta<1$ such that $\om$ is a maximal $\delta$-separated set of points in $\Sp^{n-1}$. In what follows, we denote by $\norm{T}_{L^p(\R^n)}$ the $L^p$-operator norm of an operator $T$ and write $A\les B$ to indicate that there is an absolute constant $C>0$ such that $A \leq C B$.
\begin{theorem}\label{thm:sin}
Let $n\geq 3$. Assume that $\om\subset \Sp^{n-1}$ is equidistributed. Then
\begin{align*}
\norm{M_{\om} }_{L^2(\R^n)} \les  (\# \om)^{\frac{n-2}{2(n-1)}}  \quad \text{ and } \quad
 \norm{T_{\om} }_{L^2(\R^n)} \les (\# \om)^{\frac{n-2}{2(n-1)}}.
\end{align*}
\end{theorem}

Both bounds in Theorem \ref{thm:sin} are sharp in general. To see this, one may test $M_\om$ and $H_\om$ to the characteristic function of a ball. The sharp upper bound for $\norm{H_\om}_{L^2(\R^n)}$ for equidistributed $\om$ is due to Kim \cite{Kim}. 

Before we discuss earlier results in $\R^2$, we mention a trivial bound. For each $v\in \Sp^{n-1}$, $M_v$ and $T_v$ are $L^p(\R^n)$-bounded for any $1<p\leq \infty$ and any $1<p<\infty$, respectively. This follows from the boundedness of the Hardy-Littlewood maximal function and the classical Mikhlin multiplier theorem. Using the embedding $l^p \hookrightarrow l^\infty$, one obtains a trivial bound $O((\# \om)^{1/p})$ for $\norm{M_\om}_{L^p(\R^n)}$ and $\norm{T_\om}_{L^p(\R^n)}$. The main problem is to obtain sharp bounds; it is conjectured that they are $O_\epsilon((\# \om)^{\epsilon})$ for any $\epsilon>0$ when $p=n$. See \cite{Kim, Dem0, Dem} for calculations yielding the sharpness of these bounds.

In $\R^2$, Str\"{o}mberg \cite{Str} obtained the sharp bound $\norm{M_{\om}}_{L^2(\R^2)}= O(\log \# \om)$ for equidistributed sets of directions $\om$, improving an earlier result by Cordoba \cite{Cor,CorP}. Katz \cite{Katz} proved that the same bound is valid for arbitrary finite $\om \subset \Sp^1$. Another proof of that result was given by Alfonseca-Soria-Vargas \cite{ASV,ASV2} based on an almost-orthogonality principle for $M_\om$. See \cite{Wainger, Carbery, Alf} for relevant works, \cite[Chapter X]{Stein} for an overview of maximal averaging operators and \cite{NSW, Car_Lac, Bateman, Parcet-Rogers} for $L^p(\R^n)$-estimates for $M_\om$ when $\om$ is lacunary.

A singular integral analogue of Katz's result in $\R^2$ was obtained by Christ-Duoandikoetxea-Rubio de Francia \cite{CDR} and Kim \cite{Kim}, independently, for the maximal directional Hilbert transform $H_\om$, and by Demeter \cite{Dem0} more generally for $T_\om$. The paper \cite{Dem0} uses a weak type $(2,2)$ estimate for the maximal function $f \mapsto \sup_{v\in S^1} |T_v L_k f|$ due to Lacey-Li \cite{LL}, where $L_k$ is a Littlewood-Paley projection operator to the frequency $|\xi| \sim 2^k$. See \cite{DD,DP,DP2, ADP} for further results in $\R^2$ and sharp estimates for lacunary direction sets in $\R^n$.

In higher dimensions $\R^n$, much less is known about $\norm{M_\om}_{L^p(\R^n)}$ and $\norm{T_\om}_{L^p(\R^n)}$. Indeed, we are not aware of any sharp result for $p> 2$ when $\om$ is equidistributed. This is in contrast with advances on $L^p$-estimates for the Kakeya and Nikodym maximal functions, starting from the works \cite{CDR, Bo, Wolff}. See also the papers \cite{Bar, Dem, DPavg}, where maximal functions slightly weaker than $M_\om$ were considered. For equidistributed $\om$, these papers provide $L^2$-bounds which are sharp up to an additional logarithmic factor. 

In particular, for the single scale maximal averaging operator
\[ M^{single}_{\om} f(x) = \sup_{v\in \om}  \int_{-1/2}^{1/2} | f(x-vt)| dt, \]
Demeter \cite{Dem} and Di Plinio-Parissis \cite{DPavg} established a nearly sharp bound for \emph{arbitrary} finite sets $\om \subset \Sp^{n-1}$; for any $\epsilon>0$\,
\begin{equation}\label{eqn:dp}
 \norm{M^{single}_{\om}}_{L^2(\R^n)} \leq C_\epsilon  (\# \om)^{\frac{n-2}{2(n-1)}+\epsilon}
\end{equation}
with a logarithmic refinement for $n=3$. An analogue of \eqref{eqn:dp} for $H_\om$ was obtained by the author and Pramanik \cite{KP}. In addition, the papers \cite{DPavg, KP} show that bounds for $M^{single}_\om$ and $H_\om$ can be improved when $\om$ is contained in a subvariety of $\Sp^{n-1}$. It would be interesting to extend Theorem \ref{thm:sin} so that it includes the above mentioned results. 

The proof of Theorem \ref{thm:sin} is by induction on the scale $\delta$ based on almost-orthogonality principles, Theorem \ref{thm:ortho} below. For the almost-orthogonality principle, it is convenient to work with a variant of $M_\om$. We fix a smooth function $\phi$ such that $\phi \ges \ind_{[-1,1]}$ and $\wh{\phi}$ is supported on $[-1,1]$. For each $v\in \Sp^{n-1}$ and $h>0$, we consider the averaging operator 
\[ A_{v,h} f(x) = \frac{1}{h}  \int f(x-vt) \phi(t/h) dt \]
and the associated maximal function
\[ \M_{\om} f(x) = \sup_{v\in \om}   \M_{v} f(x), \;\; \text{ where } \; \M_v f(x) = \sup_{h>0} |A_{v,h} f(x)|. \]
It follows that $M_\om f(x) \les \M_{\om} |f|(x)$, so it is sufficient to study $\norm{\M_{\om}}_{L^2}$ for the proof of Theorem \ref{thm:sin}. 

A basic setup common to almost-orthogonality principles for both $\M_\om$ and $T_\om$ is the following. Given $\om \subset \Sp^{n-1}$, let $\{O_j\}$ be subsets of $\Sp^{n-1}$  covering $\om$ with the diameter $d(O_j) = \delta_j\leq 1$. We let $\om_j = \om \cap O_j$ and denote by $\cO = \{v_j \}$ a collection consisting of exactly one $v_j$ from each $O_j$. In the following, we fix a constant $0<c \leq 1$ and allow implicit constants depend on $c$. For each $l\geq 0$, let
\[ E_l = E_l(c, \{\delta_j\})= \sup_{w \in \Sp^{n-1}} \# \{ j: \dist(v_j,w^{\perp}) \leq (1+c)2^l \delta_j \},\] 
where $w^{\perp} = \{ \xi \in \R^n : w \cdot \xi = 0 \}$. 

\begin{theorem} \label{thm:ortho}
Let $n\geq 2$ and $\om$, $\{O_j\}$, $\cO$ and $E_l$ be as above. Then
\begin{align}
\label{eqn:orthoM}
\norm{\M_\om}_{L^2} &\les \norm{\M_\cO}_{L^2} +  \sup_j \norm{ \M_{\om_j} }_{L^2}\sqrt{E_0}, \\
\label{eqn:orthoT}
 \norm{T_\om}_{L^2} &\leq \norm{T_\cO}_{L^2} + C \sup_j (\norm{T_{\om_j} }_{L^2}+ \sqrt{ \log \# \om_j }) \sum_{l\geq 0} 2^{-l} \sqrt{E_l}. 
\end{align}

\end{theorem}

For a maximal $\delta$-separated set of points $\om$ in $\mathbb{S}^1$, \eqref{eqn:orthoT} recovers the sharp $O(\log ( \# \om))$ bound for $\norm{T_\om}_{L^2}$ from \cite{Dem0} by a simple induction argument. Indeed, one can choose disjoint arcs $\{ O_j \}$ covering $\om$ so that $\# \om_j \sim 1$ and $\cO$ is a maximal $5\delta$-separated set of points in $\mathbb{S}^1$. We recall that the previously known proof \cite{Dem0} relies on a result \cite{LL} which is known to imply Carleson's theorem on the pointwise convergence of Fourier series.

In addition to Theorem \ref{thm:sin}, Theorem \ref{thm:ortho} may yield new results for certain non-equidistributed direction sets which have both lacunary and equidistributed features, using known results for lacunary direction sets \cite{Parcet-Rogers, ADP}. We state it as a corollary.
\begin{cor} \label{cor:mixed} Let $n\geq 3$ and $0<\delta<1$. Consider a finitely many overlapping  cover $\{O_j\}$ of $\Sp^{n-1}$ consisting of caps of diameter $\delta$. Suppose that $\om = \cup_j \om_j$ and $\om_j$ is a lacunary direction set (as defined in \cite{Parcet-Rogers}) contained in $O_j$.
Then
\begin{align*}
\norm{M_{\om} }_{L^2(\R^n)} &\les  \delta^{-\frac{n-2}{2}}  \\
 \norm{T_{\om} }_{L^2(\R^n)} &\les \delta^{-\frac{n-2}{2}} \max_j \sqrt{ \log(\# \om_j) }.
\end{align*}

\end{cor}

The almost orthogonality principle \eqref{eqn:orthoM} for $\M_\om$ can be regarded as a generalization of a result due to Alfonseca \cite{Alf} for the case $n=2$. The proof of \eqref{eqn:orthoM} is based on \cite{Alf} and also the work of Duoandikoetxea and Moyua \cite{DM}. For the proof of \eqref{eqn:orthoT}, we start by writing $T_v = T_{v_j} + [T_v- T_{v_j}]$ for $v \in \om_j$. When $m(\xi) = -i\sgn(\xi)$, the multiplier for the difference $T_v - T_{v_j}$ is supported in a conic region determined by $v$ and $v_j$, and this fact played a crucial role in the papers \cite{Kim, KP}. However, this localization property fails in general. To handle the term $T_v - T_{v_j}$, we break the frequency space into conic regions according to the size of $|m(v\cdot \xi) - m(v_j\cdot \xi)|$. In addition, we use a square-function reduction which is responsible for the term $\sqrt{\log \# \om_j}$ in \eqref{eqn:orthoT}; see Proposition \ref{prop:CWW}. 

A weakness of Theorem \ref{thm:ortho} is that the numbers $E_l$ depend on the diameters of $\{O_j\}$, which originates from our choice of the decomposition for the frequency space mentioned earlier. Theorem \ref{thm:ortho} seems particularly weak when $\om$ is contained in a subvariety of $\Sp^{n-1}$ and the sets $\{O_j \}$ are subsets of the subvariety. For such a lower-dimensional situation, a more refined decomposition of the frequency space might be useful, but we do not explore it in this paper.

\section{Proof of Theorem \ref{thm:sin} and Corollary \ref{cor:mixed}} \label{sec:proofmain}
In this section, we prove Theorem \ref{thm:sin} and Corollary \ref{cor:mixed} assuming Theorem \ref{thm:ortho}. We prove Theorem \ref{thm:ortho} in the following sections.

\subsection{Proof of Theorem \ref{thm:sin}}
Assume that $\om$ is a maximal $\delta$-separated set in $\Sp^{n-1}$ for some $0<\delta<1$. Since $\# \om \sim \delta^{-(n-1)}$, it suffices to show that
\begin{equation}\label{eqn:main1}
\norm{\M_{\om}}_{L^2} \les \delta^{-(n-2)/2} \;\; \text{ and } \;\; \norm{T_{\om}}_{L^2} \les \delta^{-(n-2)/2}.
\end{equation}
The proof of \eqref{eqn:main1} is essentially the same for $\M_\om$ and $T_\om$. We will prove the statement for $\M_\om$ and mention minor modifications for $T_\om$. 

For the purpose of an induction argument, we consider a slightly more general statement. For $0<\delta\leq \eta\leq 1$, let $\Lambda_{\delta,\eta}$ be the collection of all $\delta$-separated subsets $\om$ of $\Sp^{n-1}$ such that $\operatorname{diameter} (\om) = \max_{v,v'\in \om} |v-v'| \leq \eta$. For $\om \in \Lambda_{\delta,\eta}$, note that $\# \om \les (\eta/\delta)^{n-1}$. Define
\[ C(\delta,\eta) = \sup_{\om \in \Lambda_{\delta,\eta}} \norm{\M_{\om}}_{L^2}. \]
We claim that there is an absolute constant $A$ such that
\begin{equation}\label{eqn:claim}
C(\delta,\eta) \leq A (\eta/\delta)^{(n-2)/2},
\end{equation}
so that \eqref{eqn:main1} is a special case of \eqref{eqn:claim} with $\eta = 1$. 

We prove the claim \eqref{eqn:claim} by induction on $\eta/\delta$. The base of the induction is the case $\eta/\delta \sim 1$, which holds trivially for sufficiently large $A$.

Now suppose that $\om \in \Lambda_{\delta,\eta}$ for some $0 < \delta \ll \eta \leq 1$. Let $C_0$ be the implicit constant in Theorem \ref{thm:ortho}. Recall that $n\geq 3$. We choose a constant $C \geq 2$ so that $C_0 C^{-(n-2)/2} \leq 1/2$. We fix a maximal $C \delta$-separated set of points $\cO=\{v_j\} $ from $\om$. Since $\cO \in \Lambda_{C\delta, \eta}$, by induction, we have 
\begin{equation}\label{eqn:co}
 C_0 \norm{M_{\cO}}_{L^2}  \leq  C_0 A(\eta/C\delta)^{(n-2)/2} \leq  A (\eta/\delta)^{(n-2)/2} /2.
\end{equation}

We apply Theorem \ref{thm:ortho} with caps $\{ O_j \}$, where $O_j$ is the intersection of $\Sp^{n-1}$ and the ball of radius $C\delta$ centered at $v_j$. We claim that if we choose $A$ sufficiently large, then 
\begin{equation}\label{eqn:corred}
 C_0  \sup_j \norm{ \M_{\om_j} }_{L^2} \sqrt{E_0} \leq A  (\eta/\delta)^{(n-2)/2}/2.
\end{equation}
Given \eqref{eqn:co} and \eqref{eqn:corred}, Theorem \ref{thm:ortho} yields \eqref{eqn:claim}, closing the induction. To see \eqref{eqn:corred}, first note that by the trivial estimate, $\norm{ \M_{\om_j} }_{L^2} \les \# \om_j \les C^{n-1}$. Next, we estimate $E_0$. Since $\{ v_j\} \subset \Sp^{n-1}$ are $C\delta$-separated points in a ball of radius $\eta$, for any plane $w^\perp$, there are at most $\les (\eta/C\delta)^{n-2}$ many $v_j$ such that $\dist(v_j, w^\perp) \leq 2 C\delta$. Therefore, 
\begin{equation}\label{eqn:e0}
E_0 \les (\eta/C\delta)^{n-2},
\end{equation}
 Combining the estimates for $E_0$ and $\norm{ \M_{\om_j} }_{L^2}$, we see that \eqref{eqn:corred} holds provided that $A$ is sufficiently large. 

The proof for $T_\om$ is similar except for the computation in \eqref{eqn:corred}; the use of \eqref{eqn:orthoT} leads us to consider $E_l$ for $l\geq 1$. Arguing as in the estimation of $E_0$, we compute that $E_l \les 2^l (\eta/C\delta)^{n-2}$. Therefore, 
\begin{equation}\label{eqn:el}
\sum_{l \geq 0} 2^{-l} \sqrt{E_l} \les (\eta/C\delta)^{(n-2)/2}, 
\end{equation}
which gives a version of \eqref{eqn:corred} for $T_\om$. This finishes the proof.

\subsection{Proof of Corollary \ref{cor:mixed}}
We fix a point $v_j$ from $O_j$. Without loss of generality, we may assume that the points in $\cO:=\{v_j\}$ are $\delta$-separated. Since $\# \cO \les \delta^{-(n-1)},$
 Theorem \ref{thm:sin} implies that
 \begin{align*}
\norm{M_{\cO} }_{L^2(\R^n)} \les  \delta^{-\frac{n-2}{2}}  \quad \text{ and } \quad
 \norm{T_{\cO} }_{L^2(\R^n)} \les \delta^{-\frac{n-2}{2}}.
\end{align*}
For lacunary $\om_j$, the papers \cite{Parcet-Rogers, ADP} give
 \begin{align*}
\norm{M_{\om_j} }_{L^2(\R^n)} \les  1  \quad \text{ and } \quad
 \norm{T_{\om_j} }_{L^2(\R^n)} \les \sqrt{ \log(\# \om_j )}.
\end{align*}
Corollary \ref{cor:mixed} follows from Theorem \ref{thm:ortho}, the above estimates, together with \eqref{eqn:e0} and \eqref{eqn:el} with $\eta=1$.

\section{Almost orthogonality principle for $\M_\om$: Proof of \eqref{eqn:orthoM}} \label{sec:M}
We will prove \eqref{eqn:orthoM} with $c=1$; the case $0<c<1$ requires only obvious modifications. 

We consider a Nikodym-type maximal function 
\[ \Ny_{v, \delta} f(x) = \sup_{R \in B(v, \delta)} \frac{1}{|R|} \int_{R}  |f(x-y)| dy, \;\; \Ny_{\om, \delta} f(x) = \sup_{v \in \om} \Ny_{v,\delta} f(x), \]
where $0<\delta<1$ and $B(v, \delta)$ denotes the collection of all rectangles of dimensions $h(1 \times \delta \times \cdots \times \delta)$ for all $h>0$ pointing in the direction $v$, centered at the origin. 

As is well-known, $\M_{v}$ controls $\Ny_{v,\delta}$ up to a composition with the Hardy-Littlewood maximal function. Without loss of generality, we may assume that $|v-e_n| < 1/100$ for every $v\in \om$. Then we have 
\begin{equation} \label{eqn:nykMv}
\Ny_{v,\delta} f(x) \les  \M_{v} M_{HL'} f(x)
\end{equation}
uniformly in $\delta$, where $M_{HL'}$ denotes the Hardy-Littlewood maximal function acting only on the first $(n-1)$-variables. There is a converse statement for functions with compact Fourier supports. Let $\vp$ be a smooth radial function supported on $|\xi|\leq 2$ such that $\vp(\xi) = 1$ for $|\xi| \leq 1$. We denote by $\vp(D)$ the multiplier transform $\widehat{\vp(D) f }(\xi ) = \vp(\xi) \wh{f}(\xi)$. 

\begin{lem}[cf. {\cite[Lemma 3]{DM}}] \label{lem:nyk}
\[ \sup_{h>0} |A_{v,h} \vp(h\delta D) f(x)| \les \Ny_{v,\delta} f(x).\]
\end{lem}
\begin{proof}
Note that the Fourier multiplier for $A_{v,h} \vp(h\delta D)$ is $\wh{\phi}(h v\cdot \xi) \vp(h\delta \xi)$, which is a bump function on a rectangle of dimensions $h^{-1} \times h^{-1} \delta^{-1} \times \cdots \times h^{-1} \delta^{-1}$ with the short direction in $v$. Thus, $A_{v,h} \vp(h\delta D)$ is essentially an averaging operator over dual rectangles of dimensions $h(1 \times \delta \times \cdots \times \delta)$ in $B(v,\delta)$ and is controlled by $\Ny_{v,\delta}$. We omit the details of this standard computation.
\end{proof}

Next, we consider the remaining part: $A_{v,h} (I-\vp(h\delta D))$. The Fourier multiplier of the remaining part is supported in \[ \{ \xi: |h v\cdot \xi| \leq 1, \; |h\delta \xi| \geq 1 \} \subset  \{ \xi: |v \cdot \xi| \leq \delta  |\xi| \} . \]
Let $R_{W}$ be the Fourier restriction operator to the set $W$; \[ \wh{R_{W} f}(\xi) = \ind_W(\xi) \wh{f}(\xi).\] By the above observation, we have 
\begin{equation}\label{eqn:supp}
A_{v,h} (I-\vp(h\delta D)) = A_{v,h} (I-\vp(h\delta D)) R_W 
\end{equation}
for any $W$ containing 
\begin{equation}\label{eqn:Cvd}
C_{v,\delta} := \{ \xi \in \R^n: |v \cdot \xi| \leq \delta  |\xi| \}. 
\end{equation}

\begin{lem}[cf. {\cite[Lemma 3]{DM}}] \label{lem:sqr}  If $C_{v,\delta} \subset W$, then for a.e. $x$, 
\[ \sup_{h>0} |A_{v,h} (I-\vp(h\delta D)) f(x)| \les \M_v M_{HL} R_W f (x). \]
\end{lem}
\begin{proof}
This follows from \eqref{eqn:supp} and the fact that $ |(I-\vp(h\delta D))  f(x)| \les M_{HL} f(x)$ for a.e. $x$ for the Hardy-Littlewood maximal function $M_{HL}$.
\end{proof}

\begin{lem}[cf. {\cite[Lemma 3]{Alf}}] 
Let $v, v_j \in \Sp^{n-1}$ such that $|v -v_j| \leq \delta_j$ and $W_j=C_{v_j,2\delta_j}$. Then for almost every $x$,
\begin{equation}\label{eqn:pt}
\M_v f(x) \les \Ny_{v_j,\delta_j} f(x) + \M_v M_{HL} R_{W_j} f (x).
\end{equation} 
\end{lem}
\begin{proof}
We first note that $\Ny_{v_j,\delta_j} f(x)$ is comparable to $\Ny_{v,\delta_j} f(x)$ as $|v-v_j| \leq \delta_j$. Moreover, by the triangle inequality, we observe that $C_{v,\delta_j} \subset W_j.$
Thus, the estimate is an application of Lemma \ref{lem:nyk} and Lemma \ref{lem:sqr} with $\delta=\delta_j$. 
\end{proof}

\begin{proof}[Proof of Theorem \ref{thm:ortho}]
We take sup over $v\in \om_j$ in \eqref{eqn:pt} and then take sup over $j$. This gives 
\[ 
\M_\om f(x) \les \sup_j \Ny_{v_j,\delta_j} f(x) + \sup_j \M_{\om_j} M_{HL} R_{W_j} f (x).
\]
Using the pointwise estimate \eqref{eqn:nykMv} and the embedding $l^2 \hookrightarrow l^\infty$, we get
\[ 
\M_\om f(x) \les \M_{\cO} M_{HL'}f(x) + \left(\sum_j |\M_{\om_j} M_{HL} R_{W_j} f (x)|^2 \right)^{1/2}.
\]
Using the $L^2$ boundedness of the Hardy-Littlewood maximal function and Plancherel, we see that
\[ 
\norm{\M_\om }_{L^2} \les \norm{\M_{\cO}}_{L^2} + \norm{ \sum_j \ind_{W_j } }_{L^\infty}^{1/2} \sup_j \norm{\M_{\om_j}}_{L^2}.
\]
Thus, it remains to examine $\sum_j \ind_{ W_j }(w)$. Since the function is homogeneous of degree 0, we may assume that $w\in \Sp^{n-1}$. Recall that $W_j=C_{v_j,2\delta_j}$, so $w\in W_j$ means that $|v_j\cdot w| \leq 2\delta_j$, or equivalently $\dist(v_j,w^\perp) \leq 2\delta_j$. Therefore, $\norm{ \sum_j \ind_{W_j } }_{L^\infty}$ is bounded by $E_0$ and this finishes the proof.
\end{proof}

\section{Almost orthogonality principle for $T_\om$: Proof of \eqref{eqn:orthoT}} \label{sec:T}
We consider conic regions $C_{v,\delta}$ defined in \eqref{eqn:Cvd}. For the given $0<c\leq 1$ and each $j$, we partition $\R^n$ into conic regions
\begin{align*}
W^{0}_j &= C_{v_j,(1+c)\delta_j} \\
W^{l}_j &= C_{v_j,(1+c)\delta_j 2^l} \setminus C_{v_j,(1+c)\delta_j 2^{l-1}}, \;\; \text{ for } \;  l\geq 1.
\end{align*}
 For each $j$, let $L_j$ be the smallest $l$ for which $(1+c) \delta_j 2^l \geq 1$. We have $C_{v_j,(1+c)\delta_j 2^{l}} = \R^n$ for $l > L_j$ by the choice of $L_j$, so $W_j^{l}$ is empty when $l>L_j$. 

Let $v \in \om_j$. We decompose
\begin{equation}\label{eqn:dec}
 T_v = T_{v_j} + [T_v - T_{v_j}]=T_{v_j} + [T_v-T_{v_j}] R_{W^0_j} + \sum_{l\geq 1} [T_v-T_{v_j}] R_{W^l_j}.
\end{equation}
This decomposition is motivated by the following pointwise bound
\begin{equation}\label{eqn:diff}
|m(v\cdot \xi) - m(v_j \cdot \xi)| \ind_{W^l_j}(\xi) \les 2^{-l}
\end{equation}
which holds uniformly for all $v\in \Sp^{n-1}$ with $|v-v_j|\les \delta_j$.
Indeed, we can write 
\begin{equation}\label{eqn:Wlj}
W^{l}_j = \{ \xi : (1+c)2^{l-1} \delta_j |\xi| < |v_j \cdot \xi| \leq (1+c)2^l \delta_j |\xi| \},
\end{equation}
and for $\xi \in W^l_j$, $|m(v\cdot \xi) - m(v_j \cdot \xi)|$ is bounded by 
$ C |(v-v_j) \cdot \xi|  / |v_j \cdot \xi| \les \delta_j |\xi| / ( 2^{l} \delta_j |\xi|) = 2^{-l}.$

From \eqref{eqn:dec}, we get the pointwise estimate
\[ T_\om f \leq T_\cO f + \sup_j T_{\om_j} R_{W^0_j} f + \sup_j |T_{v_j} R_{W^0_j} f| + \sum_{l \geq 1} \sup_j \sup_{v\in \om_j}| [T_v-T_{v_j}] R_{W^l_j} f|. \]
From this, as in the proof of \eqref{eqn:orthoM}, we may obtain
\begin{equation}\label{eqn:mid}
\begin{split}
 \norm{T_\om f}_2 &\leq \bigg( \norm{T_\cO}_2 + (\sup_j \norm{T_{\om_j} }_2+ C_0) \norm{ \sum_{j} \ind_{W^0_j} }_{\infty}^{1/2} \bigg) \norm{f}_2 \\
&+ \sum_{l \geq 1} \bigg( \sum_j \norm{ \sup_{v\in \om_j} |[T_v-T_{v_j}] R_{W^l_j}f|}_2^2 \bigg)^{1/2}, 
\end{split}
\end{equation}
where $C_0 = \norm{m}_{L^\infty}$.

Fix $j$ and $l \geq 1$. Let $L_k$ be a Littlewood-Paley frequency cut-off to the annulus $|\xi| \sim 2^k$. We handle the last display with the following square function reduction. 
\begin{prop}\label{prop:CWW} Let $\{\cT_v\}_{v\in V}$ be a finite collection of Fourier multiplier transformations bounded on $L^p(\R^n)$ for some $1<p<\infty$. Then 
\[ \norm{ \max_{v\in V} |\cT_v f| }_{p} \les (\log \# V)^{1/2} \norm{ \big(\sum_k \max_{v\in V} |\cT_v L_k f|^2\big)^{1/2} }_{p} +  \max_{v\in V} \norm{ \cT_v f }_{p}.\]
\end{prop}
This square function reduction has been used earlier for the study of directional maximal functions; see e.g. \cite{Dem0,Dem, DD}. Proposition \ref{prop:CWW} follows from the Chang-Wilson-Wolff inequality \cite{CWW} (see also \cite[Proposition 3.1]{GuoRoosSeegerYung}) and ideas from the paper \cite{GHS}, especially, Section 4.

The pointwise estimate \eqref{eqn:diff} yields
\[ \norm{[T_v-T_{v_j}] R_{W^l_j}f}_2 \les 2^{-l} \norm{R_{W^l_j}f}_2. \]
This bound and Proposition \ref{prop:CWW}, with $\cT_v = [T_v-T_{v_j}] R_{W^l_j}$ and $V = \om_j$, give
\begin{equation}\label{eqn:CWW}
\begin{split}
  &\norm{ \sup_{v\in \om_j} |[T_v-T_{v_j}] R_{W^l_j}f|}_2 \\ 
  &\les  \sqrt{ \log \# \om_j } \bigg( \sum_{k\in \Z}  \norm{ \sup_{v\in \om_j} |[T_v-T_{v_j}] R_{W^l_j} L_k f|}_2^2 \bigg)^{1/2} 
   + 2^{-l} \norm{R_{W^l_j} f}_2.  
\end{split}
\end{equation}

We claim that there is a pointwise bound 
\begin{lem} \label{lem:pt} Let $v, v_j \in \Sp^{n-1}$. If $|v -v_j| \leq \delta_j$, then
\[  |[T_v-T_{v_j}] R_{W^l_j} L_k f(x)| \les 2^{-l} M^{str}_{v_j} R_{W^l_j} L_k f (x). \]
\end{lem}
Here, $M^{str}_{v_j}$ is a strong maximal function composed with a rotation associated with $v_j$. From the $L^p$-boundedness of the strong maximal function, Lemma \ref{lem:pt} yields 
\[ \sum_{k\in \Z}\norm{\sup_{v\in \om_j} |[T_v-T_{v_j}] R_{W^l_j} L_k f}_{L^2}^2 \les 2^{-2l} \sum_{k\in \Z} \norm{ R_{W^l_j} L_k f}_{L^2}^2 \les 2^{-2l} \norm{ R_{W^l_j} f}_{L^2}^2. \]
Plugging this to \eqref{eqn:CWW}, we get 
\[  \norm{ \sup_{v\in \om_j} |[T_v-T_{v_j}] R_{W^l_j}f|}_2 \les \sqrt{ \log \# \om_j } 2^{-l} \norm{ R_{W^l_j} f}_{L^2}.\]

Thus, 
\begin{align*}
\bigg( \sum_j \norm{ \sup_{v\in \om_j} |[T_v-T_{v_j}] R_{W^l_j}f|}_2^2 \bigg)^{1/2} &\les \sup_j \sqrt{ \log \# \om_j } 2^{-l} \bigg( \sum_{j}  \norm{ R_{W^l_j} f}_2^2 \bigg)^{1/2}\\
&\leq \sup_j \sqrt{ \log \# \om_j } 2^{-l} \norm{ \sum_{j} \ind_{W^l_j} }_{\infty}^{1/2}  \norm{ f}_2.
\end{align*}
Combining this with \eqref{eqn:mid}, we obtain
\[
 \norm{T_\om}_{L^2} \leq \norm{T_\cO}_{L^2} + C \sup_j (\norm{T_{\om_j} }_{L^2}+ \sqrt{ \log \# \om_j }) \sum_{l\geq 0} 2^{-l} \norm{ \sum_{j} \ind_{W^l_j} }_{\infty}^{1/2}.  \]
Arguing as in the proof of \eqref{eqn:orthoM}, one verifies that 
\[ \norm{ \sum_{j} \ind_{W^l_j} }_{\infty} \leq E_l, \]
which gives \eqref{eqn:orthoT}. It only remains to prove Lemma \ref{lem:pt}.

\begin{proof}[Proof of Lemma \ref{lem:pt}]
Let $\psi$ be a smooth function whose value is 1 on the set $\{ t\in \R : (1+c)/2 \leq |t| \leq 1+c \}$ and supported on $\{ t\in \R : \frac{1}{2}+\frac{c}{4} \leq |t| \leq 1+\frac{5}{4}c \}$. Using \eqref{eqn:Wlj}, one can verify that 
\[ \ind_{W^l_j}(\xi) = \psi\left((2^l\delta_j)^{-1} \frac{v_j\cdot \xi}{|\xi|} \right) \ind_{W^l_j}(\xi).\]

Let $\chi(2^{-k} \xi)$ be a smooth compactly supported radial multiplier for a Littlewood-Paley projection operator $\tilde{L_k}$ such that $L_k=\tilde{L_k} L_k$. Then we may write 
\[ [T_v-T_{v_j}] R_{W^l_j} L_k f = K * (R_{W^l_j} L_k f) ,\] where 
\[ K(x) =  \int (m(v\cdot \xi) - m(v_j \cdot \xi) ) \psi\left((2^l\delta_j)^{-1} \frac{v_j\cdot \xi}{|\xi|} \right) \chi(2^{-k} \xi) e^{i x\cdot \xi} d\xi .\]
By a rotation, we may assume that $v_j = e_n$. For the lemma, it suffices to prove that 
\begin{equation}\label{eqn:kernel}
 |K(x)| \les 2^{-l} \frac{2^{kn} 2^l \delta_j}{(1+2^k|x_1|)^2 \cdots (1+2^k| x_{n-1}|)^2 (1+2^k 2^l \delta_j| x_n|)^2}.
\end{equation}

Let $\tilde \xi =(\xi', 2^l \delta_j  \xi_n)$ and make the change of variables $\xi=(\xi',\xi_n) \to 2^k \tilde \xi$. 
Then we may write
\[K(x)= 2^{kn}2^l \delta_j \int (m(2^k v \cdot \tilde{\xi})  - m(2^k e_n \cdot \tilde{\xi}) ) \psi(\xi_n / |\tilde \xi| ) \chi( \tilde{\xi}) e^{i 2^k \inn{x', 2^l\delta_j x_n} \cdot \xi} d\xi .\]
The integrand is supported on the part where $|\xi_n| \sim 1$ and $|\tilde \xi| \sim 1$. We write 
\[ I(\xi) := m(2^k v \cdot \tilde{\xi})  - m(2^k e_n \cdot \tilde{\xi})  = \int_0^1 2^k (v-e_n) \cdot \tilde\xi m'(2^k s(v-e_n)\cdot \tilde \xi + 2^k e_n \cdot \tilde \xi) ds. \]
Since $|(v-e_n)\cdot \tilde \xi| \les \delta_j$ and $| s(v-e_n) \cdot \txi + e_n \cdot \txi  |\sim |e_n \cdot \txi| \sim 2^l \delta_j$, we get $ | I(\xi) | \les 2^{-l} $ for $|\xi_n| \sim |\txi| \sim 1$ from the decay of $m'$. The same upper bound holds for derivatives of $I(\xi)$ and the pointwise estimate \eqref{eqn:kernel} follows from integration by parts.

\end{proof}

\section*{Acknowledgements}
The author is partially supported by a PIMS Postdoctoral Fellowship and a Discovery grant from the Natural Sciences and Engineering Research Council of Canada. He thanks Malabika Pramanik for helpful discussions. He is grateful to the referees for their comments, which have improved the exposition of this paper.

\bibliographystyle{amsplain}

\end{document}